\documentclass[12pt,twoside]{amsart}
\usepackage{amssymb}
\usepackage{a4wide}
\usepackage[latin1]{inputenc}
\usepackage[T1]{fontenc}
\usepackage{times}

\def\hpq0{h^{p,q}_{\leq 0}}
\def\Hpq0{\H_{\leq 0}^{p,q}}

\def\dt{\partial^{\phi_t}}
\def\dbar{\bar\partial}
\def\ddbar{\partial\dbar}

\def\R{{\mathbb R}}

\def\C{{\mathbb C}}

\def\G{{\mathcal G}}

\def\F{{\mathcal F}}

\def\X{{\mathcal X}}

\def\dof{\dot{\phi_t}}
\def\dofnu{\dot{\phi_t^\nu}}

\def\V{{\mathcal V}}

\def\H{{\mathcal H}}
\def\E{{\mathcal E}}

\def\Re{{\rm Re\,  }}
\def\Im{{\rm Im\,  }}

\def\be{\begin{equation}}
\def\ee{\end{equation}}

\newtheorem{thm}{Theorem}[section]
\newtheorem{lma}[thm]{Lemma}

\newtheorem{prop}[thm]{Proposition}

\theoremstyle{definition}

\theoremstyle{remark}

\newtheorem{preremark}{Remark}
\newtheorem{preex}{Example}

\newenvironment{remark}{\begin{preremark}}{\qed\end{preremark}}

\numberwithin{equation}{section}

\begin{document}

\title[]
{A Brunn-Minkowski type inequality for Fano manifolds and the
  Bando-Mabuchi uniqueness theorem.}

\address{B Berndtsson :Department of Mathematics\\Chalmers University
  of Technology \\
  and Department of Mathematics\\University of G\"oteborg\\S-412 96
  G\"OTEBORG\\SWEDEN,\\} 

\email{ bob@math.chalmers.se}

\author[]{ Bo Berndtsson}

\begin{abstract}
For $\phi$ a metric on the anticanonical bundle, $-K_X$, of a Fano manifold $X$
we consider the volume of $X$
$$
\int_X e^{-\phi}.
$$
We prove that the logarithm of the volume is concave along bounded geodesics
in the space of positively curved metrics on $-K_X$ and that the
concavity is strict unless the geodesic comes from the flow of a
holomorphic vector field on $X$. As a consequence we get a simplified
proof of the Bando-Mabuchi uniqueness theorem for K\"ahler - Einstein
metrics. We also prove  a generalization of this theorem to 'twisted'
K\"ahler-Einstein metrics and treat some classes of manifolds that
satisfy weaker hypotheses than being Fano. .
\end{abstract}

\bigskip

\maketitle

\section{Introduction}

Let $X$ be an $n$-dimensional  projective  manifold with
seminegative canonical 
bundle and let $\Omega$ be a domain in the complex plane. We consider
curves $t\rightarrow \phi_t$, with $t$ in $\Omega$,  of metrics on $-K_X$  that have
plurisubharmonic 
variation so that $i\ddbar_{t, X}\phi\geq 0$ ( see section 2 for
notational conventions). Then $\phi$ solves the
homogenous Monge-Amp\`ere equation if  
\be
(i\ddbar\phi)^{n+1}=0.
\ee
By a
fundamental theorem of Chen,  \cite{Chen}, we can for any given $\phi_0$ defined
on the boundary of $\Omega$, smooth with nonnegative curvature on $X$
for $t$ fixed on $\partial\Omega$, find a solution of (1.1)
with $\phi_0$ as boundary values. This solution does in general not
need to be smooth (see \cite{1Donaldson}), but Chen's theorem asserts
that we can find a solution that has all mixed complex derivatives
bounded, i e $\ddbar_{t, X}\phi$ is bounded on $X\times \Omega$. The
solution equals the 
supremum (or maximum) of all subsolutions, i e all metrics with
semipositive curvature that are dominated by $\phi_0$ on the
boundary. Chen's proof is based on some of  the methods from Yau's proof of the
Calabi conjecture, so it is not so easy, but it is worth pointing out
that the existence of 
a generalized solution that is only bounded is much easier, see
section 2. On the
other hand, if we assume that $\phi$ is smooth and $i\ddbar_X\phi>0$
on $X$ for any 
$t$ fixed, then
$$
(i\ddbar\phi)^{n+1}=n c(\phi)(i\ddbar\phi)^n\wedge i dt\wedge d\bar t
$$
with 
$$
c(\phi)=\frac{\partial^2\phi}{\partial t\partial\bar
  t}-|\dbar\frac{\partial\phi}{\partial t}|^2_{i\ddbar_X\phi},
$$
where the norm in the last term is the norm with respect to the K\"ahler metric
$i\ddbar_X\phi$. Thus equation 1.1 is then equivalent to $c(\phi)=0$. 

The case when $\Omega=\{t; 0<\Re t<1\}$ is a strip  is of
particular interest. If the boundary data are also independent of $\Im
t$ the solution to 1.1 has a similar invariance property. A famous
observation of Semmes, \cite{Semmes} and Donaldson, \cite{Donaldson}
is that the equation $c(\phi)=0$ then is the equation for a geodesic
in the space of K\"ahler potentials. Chen's theorem then {\it almost}
implies that 
any two points in the space of K\"ahler potentials can be joined by a 
geodesic, the proviso being that we might not be able to keep
smoothness or strict positivity along all of the curve. This problem
causes some difficulties in applications, one of which we will
address in this paper.

The next theorem is a direct consequence of the results in
\cite{2Berndtsson}.
\begin{thm}Assume that $-K_X\geq 0$ and let
let $\phi_t$ be a  curve of  metrics on $-K_X$
such that
$$
i\ddbar_{t,X}\phi\geq 0
$$
in the sense of currents. Then 
$$
\F(t):=-\log\int_X e^{-\phi_t}.
$$
 is subharmonic in $\Omega$. In particular, if $\phi_t$ does not
 depend on the imaginary part of $t$, $\F$ is convex. 
\end{thm}

Here we interpret the integral over $X$ in the following way. For any
choice of local coordinates $z^j$ in some covering of $X$ by
coordinate neighbourhoods $U_j$, the metric $\phi_t$ is
represented by a local function $\phi_t^j$. The volume form
$$
c_n e^{-\phi^j_t}  dz^j\wedge\bar dz^j,
$$
where $c_n=i^{n^2}$ is a unimodular constant chosen to make the form
positive, is independent of the choice of local coordinates. We denote this
volume form by $e^{-\phi_t}$, see section 2. 

The results in \cite{2Berndtsson}
 deal with more
general line bundles $L$  over $X$, and  the trivial vector
bundle $E$ over $\Omega$ with fiber $H^0(X, K_X+L)$ with the
$L^2$-metric
$$
\|u\|^2_t=\int_X |u|^2 e^{-\phi_t},
$$
see section 2. The main result is then a formula for the curvature of
$E$ with the $L^2$-metric. In this paper we study the simplest special
case, $L=-K_X$. Then $K_X+L$ is trivial so $E$ is a line bundle and
Theorem 1.1 says that this line bundle has nonnegative curvature. 

Theorem 1.1 is formally analogous to the Brunn-Minkowski inequality
for the volumes of convex sets, and even more to its functional version,
Prekopa's theorem, \cite{Prekopa}. Prekopa's theorem states that if
$\phi$ is a convex function on $\R^{n+1}$, 
then
$$
f(t):=-\log\int_{\R^n}e^{-\phi_t}
$$
is convex. The complex counterpart of this is that we consider a
complex manifold $X$ with a family of volume forms $\mu_t$. In local
coordinates $z^j$  the volume form can be written as above
$c_n e^{-\phi^j_t} dz^j\wedge\bar dz^j$ , and if $\mu_t$ is
globally well defined $\phi^j_t$ are then the local representatives of
a metric, $\phi_t$, on $-K_X$. Convexity in Prekopa's theorem then
corresponds to 
positive, or at least semipositive, curvature of $\phi_t$, so $X$ must
be Fano, or its canonical bundle must have at least have seminegative
curvature  (in some sense:
$-K_X$ pseudoeffective would be the minimal requirement). The
assumption in Prekopa's theorem that the weight is convex with respect
to $x$ and $t$ together then correspond to the assumptions in Theorem
1.1.

If $K$ is a compact convex set in $\R^{n+1}$ we can take $\phi$ to
be equal to 0 in $K$ and $+\infty$ outside of $K$. Prekopa's theorem
then implies the Brunn-Minkowski theorem, saying that the logarithm of
the volumes of $n$-dimensional slices, $K_t$ of convex sets are
concave; concretely 
\be
|K_{(t+s)/2}|^2\leq |K_t||K_s|
\ee

The Brunn-Minkowski theorem has an important addendum which describes
the case of equality : If equality holds in (1.2)
then all the slices $K_t$ and $K_s$  are translates of each other
$$
K_t=K_s + (t-s)\mathbf v
$$
where $\mathbf v$ is some  vector in $\R^n$.
A little bit
artificially we can formulate this as saying that we move from one
slice to another via the flow of a constant vector field. 

\begin{remark}
It follows that from (1.2) and the natural homogenuity properties of
Lebesgue measure that $|K_t|^{1/n}$, is also concave. This ('additive
version')  is perhaps
the most common formulation of the Brunn-Minkowski inequalities, but
the logarithmic (or multiplicative) version above works better for
weighted volumes and in the complex setting.
 For the additive version
conditions for equality are more liberal; then $K_t$ may change not
only by translation but also by dilation 
(see \cite{Gardner}), but equality in the multiplicative case excludes
dilation. 
\end{remark}

A natural
question is then if one can draw a similar conclusion in the complex
setting described above. In \cite{2Berndtsson} we proved that this is
indeed so if $\phi$ is known to be smooth and strictly
plurisubharmonic on $X$ for $t$ fixed. The main result of this paper
is the extension of this to less regular situations. We keep the same
assumptions as in Theorem 1.1. 
\begin{thm}Assume that $H^{0,1}(X)=0$, and that the curve of metrics
 $\phi_t$ is 
  independent of the imaginary part of $t$.  Assume moreover that the
  metrics $\phi_t$ are uniformly bounded in the sense that for some
  smooth metric on $-K_X$, $\psi$,
$$
|\phi_t-\psi|\leq C.
$$
Then, if the function $\F$ in Theorem 1.1 is affine in a neighbourhood of
0 in $\Omega$, 
there is a (possibly time dependent) holomorphic vector field $V$ on
$X$ with flow $F_t$ such 
that
$$
F_t^*(\ddbar\phi_t) =\ddbar\phi_0.
$$
\end{thm}
The same conclusion can also be drawn without the assumption that
$\phi_t$ be independent of the imaginary part of $t$, and then
assuming that $\F$ be harmonic instead of affine,  but the proof
then seems to require more regularity assumptions. For simplicity we
therefore treat only the case when $\phi_t$ is independent of $t$,
which anyway seems to be the most useful in applications. 
\bigskip

\noindent This theorem is useful in view of the discussion above on the possible lack of
regularity of geodesics.As we shall
see in section 2 the existence of a generalized geodesic satisfying
the boundedness assumption in Theorem 1.2 is almost trivial.  One
motivation for the theorem is to give a new proof of the Bando-Mabuchi
uniqueness theorem for K\"ahler-Einstein metrics on Fano
manifolds. Recall that a metric $\omega_\psi=i\ddbar\psi$, with
$\psi$ a metric on $-K_X$ solves the  K\"ahler-Einstein equation if 
$$
\text{Ric}(\omega_\psi)=\omega_\psi
$$
or equivalently if for some positive $a$
$$
e^{-\psi}=a(i\ddbar\psi)^n,
$$
where we use the convention above to interpret $e^{-\psi}$ as a volume
form. By a celebrated theorem of Bando and Mabuchi any two
K\"ahler-Einstein metrics $i\ddbar\phi_0$ and 
$i\ddbar\phi_1$ are related via the time-one flow of a holomorphic vector
field. In section 4 we shall give a proof of this fact by
joining $\phi_0$ and $\phi_1$ by a geodesic and applying Theorem
1.2.

It should be noted that a similar proof of the  Bando-Mabuchi
theorem has already been given by
Berman, \cite{Berman}. The difference between his proof and ours is
that he uses the weaker version of Theorem 1.2 from
\cite{2Berndtsson}. He then needs to prove that the geodesic joining
two K\"ahler-Einstein metrics is in fact smooth, which we do not
need, and we also avoid the use  of 
Chen's theorem since we only need the existence of a bounded
geodesic.

 A minimal assumption in Theorem 1.2 would be that $e^{-\phi_t}$ be
 integrable, instead of bounded. I do not know if the theorem holds
 in this generality, but in section 6 we will consider an
 intermediate situation where $\phi_t=\chi_t +\psi$, with $\chi_t$
 bounded and $\psi$ such that $e^{-\psi}$ is integrable, so that the
 singularities don't change with $t$. Under various positivity
 assumptions we are then able to proof a version of Theorem 1.2.
 
Apart from making the problem technically simpler, this extra
assumption that  $\phi_t=\chi_t +\psi$ also introduces an additional
structure, which seems interesting in itself. 
In section 6 we use it to  give a generalization of the Bando-Mabuchi
theorem to certain 'twisted' K\"ahler-Einstein equations, 
$$
\text{Ric}(\omega)=\omega +\theta
$$
considered
in \cite{Szekelyhidi},\cite{2Berman} and  \cite{2Donaldson}. Here
$\theta$ is a fixed 
positive $(1,1)$-current, that may  e g be the current of integration
on a klt divisor. The solutions to these
equations are then not necessarily smooth and it seems to be hard to prove
uniqueness using the original methods of Bando and Mabuchi.

Another paper that is very much related to this one is \cite{Berman
  et al}, by Berman -Boucksom-Guedj-Zeriahi. There is introduced a
variational approach to Monge-Ampere equations and K\"ahler-Einstein
equations in a nonsmooth setting and a uniqueness theorem a la
Bando-Mabuchi is proved, using continuous geodesics as we  do here,
but  in a somewhat 
less general situation. I would like to thank all of these authors   for
helpful discussions, and Robert Berman in particular for
proposing the generalized Bando-Mabuchi theorem in section 6.

\section{Preliminaries}

\subsection{Notation}
Let $L$ be a line bundle over a complex manifold $X$, and let $U_j$ be a
covering of the manifold by open sets over which $L$ is locally
trivial. A section of $L$ is then represented by a collection of
complex valued functions $s_j$ on $U_j$ that are related by the
transition functions of the bundle, $s_j=g_{j k} s_k$. A metric on $L$
is given by a collection of realvalued functions $\phi^j$ on $U_j$,
related so that
$$
|s_j|^2 e^{-\phi^j}=:|s|^2 e^{-\phi}=:|s|^2_\phi
$$
is globally well defined. We will write $\phi$ for the collection
$\phi^j$, and refer to $\phi$ as the metric on $L$, although it might
be more appropriate to call $e^{-\phi}$ the metric. (Some authors call
$\phi$ the 'weight' of the metric.)

A metric $\phi$ on $L$ induces an $L^2$-metric on the adjoint bundle
$K_X+L$. A section $u$ of $K_X+L$ can be written locally as
$$
u= dz\otimes s
$$
where $dz=dz_1\wedge ...dz_n$ for some choice of local coordinates
and $s$ is a section of $L$. We let
$$
|u|^2 e^{-\phi}:= c_n dz\wedge d\bar z |s|_\phi^2;
$$
it is a volume form on $X$. The $L^2$-norm of $u$ is
$$
\|u\|^2:=\int_X |u|^2 e^{-\phi}.
$$
Note that the $L^2$ norm depends only on the metric $\phi$ on $L$ and
does not involve any choice of metric on the manifold $X$. 

In this paper we will be mainly interested in the case when $L=-K_X$
is the anticanonical bundle. Then the adjoint bundle $K_x+L$ is
trivial and is canonically isomorphic to $X\times \C$ if we have
chosen an isomorphism between $L$ and $-K_X$. This bundle then has a
canonical trivialising section, $u$ identically equal to 1. With the
notation above
$$
\|1\|^2 =\int_X |1|^2 e^{-\phi}=\int_X e^{-\phi}.
$$
This means explicitly that we interpret the volume form
$e^{-\phi}$ as 
$$
dz^j\wedge d\bar z^j e^{-\phi_j}
$$
where $e^{-\phi^j}= |(dz^j)^{-1}|_\phi^2$ is the local representative of
the metric for the frame determined by the local coordinates. Notice
that this is consistent with the conventions indicated in the
introduction. 

\subsection{Bounded geodesics}
We now consider curves $t\rightarrow \phi_t$ of metrics on the line bundle
$L$. Here $t$ is a complex parameter but we shall (almost) only look
at curves that do not depend on the imaginary part of $t$. We say that
$\phi_t$ is a subgeodesic if $\phi_t$ is upper semicontinuous and
$i\ddbar_{t, X}\phi_t\geq 0$,  so that local
representatives are plurisubharmonic with respect to $t$ and $X$
jointly. We say that $\phi_t$ is bounded if
$$
|\phi_t-\psi|\leq C
$$
for some constant $C$ and some (hence any) smooth metric on $L$. For
bounded geodesics the complex Monge-Ampere operator is well defined
and we say that $\phi_t$ is a (generalized) geodesic if
$$
(i\ddbar_{t, X} \phi_t)^{n+1}=0.
$$

\bigskip

\noindent Let $\phi_0$ and $\phi_1$ be two bounded metrics on $L$ over
$X$ satisfying
$i\ddbar\phi_{0,1}\geq 0$.  We claim that there is a bounded geodesic
$\phi_t$ defined for the real part of $t$ between 0 and 1, such that
$$
\lim_{t\rightarrow 0,1} \phi_t =\phi_{0,1}
$$
uniformly on $X$. The curve $\phi_t$ is defined by
\be
\phi_t=\sup \{\psi_t\}
\ee
where the supremum is taken over all plurisubharmonic $\psi_t$ with 
$$
 \lim_{t\rightarrow 0,1} \psi_t \leq \phi_{0,1}.
$$
To prove that $\phi_t$ defined in this way has the desired properties
we first construct a barrier 
$$
\chi_t =\text{max} (\phi_0 -A\Re t, \phi_1 +A(\Re t -1)).
$$
Clearly $\chi$ is plurisubharmonic and has the right boundary values
if $A$ is sufficiently large. Therefore the supremum in (2.1) is the
same if we restrict it to $\psi$ that are larger than $\chi$. For such
$\psi$ the onesided derivative at 0 is larger than $-A$ and the
onesided derivative at 1 is smaller than  $A$. Since we may moreover
assume that $\psi$ is independent of the imaginary part of $t$, $\psi$
is convex in $t$ so the
derivative with respect to $t$ increases, and must therefore lie
between $-A$ and $A$. Hence $\phi_t$ satisfies
$$
\phi_0 -A\Re t\leq \phi_t\leq \phi_0 +A\Re t
$$
and a similar estimate at 1. Thus $\phi_t$ has the right boundary
values uniformly. In addition, the upper semicontinuous regularization
$\phi_t^*$ 
of $\phi_t$ must satisfy the same estimate. Since $\phi_t^*$ is
plurisubharmonic it belongs to the class of competitors for $\phi_t$
and must therefore coincide  with $\phi_t$, so $\phi_t$ is
plurisubharmonic. That finally $\phi_t$ solves the homogenuous
Monge-Ampere equation follows from the fact that it is maximal with
given boundary values, see e g \cite{Guedj-Zeriahi}. 

Notice that as a byproduct of the proof we have seen that the geodesic
joining two bounded metrics is uniformly Lipschitz in $t$. This fact
will be very useful later on. 
\subsection{Approximation of metrics and subgeodesics}

In the proofs we will need to approximate our metrics that are only
bounded, and sometimes not even bounded, by smooth metrics. Since we
do not want to lose too much of the positivity of curvature
this causes some complications and we collect here some results on
approximation of metrics that we will use. An extensive treatment of
these matters can be found in \cite{Demailly}. Here we will need only
the simplest part of this theory and we also refer to
\cite{Blocki-Kolodziej} for an elementary proof of the result we
need. 

In general a singular metric $\phi$ with $i\ddbar\phi\geq 0$ can not
be approximated by a decreasing sequnece of  smooth metrics with
nonnegative curvature. A 
basic fact is however (see \cite{Blocki-Kolodziej}, Theorem 1.1) that
this is possible if the line bundle in question is positive, so that
it has some smooth metric
of {\it strictly } positive curvature. This is all we need in the main
case of a Fano manifold.

The approximation result for positive bundles  also holds for $\mathbb
Q$-line bundles; just multiply by some sufficiently divisible integer,
and even for $\R$-bundles. In this paper we will also be interested in
line bundles that are only semipositive. If $X$ is projective, as we
 assume, the basic fact above implies that we then can approximate
any singular metric with nonnegative curvature with a decreasing
sequence of smooth $\phi^\nu$s, satisfying
$$
i\ddbar\phi^\nu\geq -\epsilon_\nu\omega
$$
where $\omega$ is some K\"ahler form and $\epsilon_\nu$ tends to
zero. To see this we basically only need to apply the result above for
the positive case to the $\R$-bundle $L+\epsilon F$ where $F$ is
positive. If $\psi$ is a smooth metric with positive curvature on $F$,
we approximate $\phi+\epsilon \psi$ by smooth metrics $\chi_\nu$ with
positive curvature. Then $\phi^\nu=\chi_\nu-\epsilon \psi$ satisfies
$$
i\ddbar \phi^\nu\geq -\epsilon \omega
$$
for $\omega=i\ddbar\psi$. Then let $\epsilon$ go to zero and choose a
diagonal sequence. This sequence may not be decreasing, but an easy
argument using Dini's lemma shows that we 
may get a decreasing sequence this way.

At one point we also wish to treat a bundle that is not even
semipositive, but only effective. It then has a global holomorphic
section, $s$, and the singular metric we are interested in is
$\log|s|^2$, or some positive multiple of it. We then let $\psi$ be
any smooth metric on the bundle and approximate by
$$
\phi^\nu:=\log( |s|^2 + \nu^{-1}e^{\psi}).
$$
Explicit computation shows that $i\ddbar\phi^\nu\geq -C\omega$ where $C$
is some fixed constant. Moreover, outside any fixed neighbourhood of
the zerodivisor of $s$,
$$
i\ddbar\phi^\nu\geq -\epsilon_\nu\omega
$$
with $\epsilon_\nu$ tending to zero. This weak approximation will be
enough for our 
purposes.

\section{The smooth case}

In this section we let $L$ be a holomorphic line bundle over $X$ and
$\Omega$ be a smoothly bounded open set in $\C$. We consider
the trivial  vector bundle $E$ 
over $\Omega$ with fiber $H^0(X, K_X+L)$. Let now $\phi_t$ be a smooth
curve of 
metrics on $L$  of semipositive curvature.  For any fixed $t$,
$\phi_t$ 
 induces an
$L^2$-norm on $H^0(X, K_X+L)$ as described in the previous section
$$
\|u\|^2_t=\int_X |u|^2 e^{-\phi_t},
$$
and as $t$ varies we get an hermitian metric on the vector bundle
$E$. 

We now recall a formula for the curvature of $E$ with this metric
from \cite{1Berndtsson},\cite{3Berndtsson}. Let for each $t$ in $\Omega$ 
$$
\dt= e^{\phi_t}\partial e^{-\phi_t}=\partial-\partial \phi_t\wedge.
$$

If $\alpha$ is an $(n,1)$-form on $X$ with values in $L$, and we write
$\alpha=v\wedge \omega$, where $\omega$ is our fixed K\"ahler form on
$X$, then (modulo a sign)
$$
\dt v=\dbar^*_{\phi_t}\alpha,
$$
the adjoint of the $\dbar$-operator for the metric $\phi_t$. In
particular this means that the operator $\dt$ is well defined on
$L$-valued forms. 

This also
means that for any $t$ we can solve the equation
$$
\dt v=\eta,
$$
if $\eta$ is an $L$-valued $(n,0)$-form that is orthogonal to the
space of holomorphic $L$-valued forms (see remark 2 below). Moreover
by choosing $\alpha=v\wedge\omega$ 
orthogonal to the kernel of $\dbar^*_{\phi_t}$ we can assume that
$\alpha$ is $\dbar$-closed, so that $\dbar v\wedge \omega=0$. Hence,
with this choice, $\dbar v$ is a primitive form. 
If, as we assume from now, the
cohomology $H^{n, 1}(X,L)=0$, the $\dbar$-operator is surjective on
$\dbar$-closed forms, so
the adjoint is injective, and $v$ is uniquely determined by $\eta$. 

\begin{remark}
 The reason we can always solve this equation for $t$ and
$\phi$ 
fixed is that the $\dbar$-operator from $L$-valued $(n,0)$-forms to
$(n,1)$-forms  on $X$ has closed range. This implies that the adjoint
operator
$\dbar^*_{\phi_t}$ also has closed range and that its range is equal
to the orthogonal complement of the kernel of $\dbar$. Moreover, that
$\dbar$ has closed range means precisely that for any $(n,1)$-form in
the range of $\dbar$ we can solve the equation $\dbar f=\alpha$ with
an estimate 
$$
\|f\|\leq C\|\alpha\|
$$
and it follows from functional analysis that we then can solve $\dt
v=\eta$ with the bound
$$
\| v\|\leq C\|\eta\|
$$
where $C$ is {\it the same} constant. In case all metrics $\phi_t$ are
of equivalent size, so that $|\phi_t-\phi_{t_0}|\leq A$ it follows
that we can solve $\dt v=\eta$ with an $L^2$-estimate independent of
$t$. 
\end{remark}

Let $u_t$
be a holomorphic section of the bundle $E$ and let 
$$
\dof:=\frac{\partial\phi}{\partial t}.
$$

\bigskip

For each $t$ we now solve
\be
\dt v_t=\pi_\perp(\dof u_t),
\ee
where $\pi_\perp$ is the orthogonal projection on the orthogonal
complement of the space of holomorphic forms, with respect to the
$L^2$-norm  $\|\cdot\|_t^2$. With this choice of $v_t$ we obtain the
following formula for the curvature of $E$, see \cite{1Berndtsson},
\cite{3Berndtsson}. In the formula, $p$ stands for the natural
projection map from $X\times\Omega$ to $\Omega$ and $p_*(T)$ is the
pushforward of a differential form or current. When $T$ is a smooth
form this is the fiberwise integral of $T$. 
\begin{thm} Let $\Theta$ be the curvature form on $E$ and let $u_t$ be
  a holomorphic section of $E$. For each $t$ in
  $\Omega$ let $v_t$ solve (3.1) and be such that $\dbar_X
  v_t\wedge \omega=0$. 
Put
$$
\hat u=u_t-dt\wedge v_t.
$$
Then
\be
\langle\Theta u_t,u_t\rangle_t= p_*(c_n i\ddbar\phi \wedge \hat
u\wedge\overline{\hat u} e^{-\phi})  
+\int_X \|\dbar v_t\|^2 e^{-\phi_t}idt\wedge d\bar t.
\ee
\end{thm}

\begin{remark} This is not quite the same formula as the one used in
  \cite{2Berndtsson} which can be seen as corresponding to a different
  choice of $v_t$.
\end{remark}

If the curvature acting on $u_t$ vanishes it follows that both terms
in the right hand side of (3.2) vanish. In particular, $v_t$ must be a
holomorphic form. To continue from there we first assume (like in
\cite{2Berndtsson})  that
$i\ddbar\phi_t>0$ on $X$. Taking $\dbar$ of formula 3.1 we get 
$$
\dbar\dt v_t=\dbar\dof\wedge u_t.
$$
Using
$$
\dbar\dt +\dt\dbar=\ddbar\phi_t
$$
we get if $v_t$ is holomorphic that
$$
\ddbar\phi_t\wedge v_t=\dbar\dof\wedge u_t.
$$
The complex gradient of the function $i\dof$ with respect to the K\"ahler
metric $i\ddbar\phi_t$ is the $(1,0)$-vector
field defined by
$$
V_t\rfloor i\ddbar\phi_t=i\dbar\dof.
$$
Since $\ddbar\phi_t\wedge u_t=0$ for bidegree reasons we get
\be
\ddbar\phi_t\wedge v_t=\dbar\dof\wedge
u=(V_t\rfloor\ddbar\phi_t)\wedge u=
-\ddbar\phi_t\wedge (V_t\rfloor u).
\ee
If $i\ddbar\phi_t>0$ we find that
$$
-v_t=V_t\rfloor u.
$$
If $v_t$ is holomorphic it follows that  $V_t$ is a holomorphic vector
field - outside of the zerodivisor of $u_t$ and therefore everywhere
since the complex gradient is smooth under our hypotheses. If we
assume that $X$ carries no nontrivial holomorphic vector fields, $V_t$
and hence $v_t$ must vanish so $\dof$ is holomorphic, hence constant. 
Hence
$$
\ddbar\dof=0
$$
so $\ddbar\phi_t$ is independent of $t$. 
In general - if there are nontrivial holomorphic vector fields - we
get that the Lie derivative of $\ddbar\phi_t$ equals
$$
L_{V_t}\ddbar\phi_t=\partial
V_t\rfloor\ddbar\phi_t=\ddbar\dof=\frac{\partial}{\partial
  t}\ddbar\phi_t .
$$
Together with an additional argument showing that $V_t$ must be
holomorphic with respect to $t$ as well (see below) this gives that
$\ddbar\phi_t$ moves with the flow of the holomorphic vector field
which is what we want to prove.

For this  it is essential that the metrics $\phi_t$ be
strictly positive on $X$ for $t$ fixed, but we shall now see that
there is a way to get around this difficulty, at least in some special
cases.

The main case that we will consider is when the canonical bundle
of $X$ is seminegative, so we can take $L=-K_X$. Then $K_X+L$ is the
trivial bundle and we fix  a nonvanishing trivializing section
$u=1$. Then the constant section $t\rightarrow u_t=u$ is  a trivializing
section of the (line) bundle $E$. We write
$$
\F(t)=-\log \|u\|_t^2=-\log\int_X |u|^2 e^{-\phi_t}=-\log\int_X e^{-\phi_t}.
$$
Still assuming that $\phi$ is smooth, but perhaps not strictly
positive on $X$, we can apply  the curvature formula in Theorem 3.1
with $u_t=u$
and get
$$
\|u_t\|^2_t i\ddbar_t\F=\langle\Theta u_t,u_t\rangle_t= p_*(c_n
i\ddbar\phi \wedge \hat 
u\wedge\overline{\hat u} e^{-\phi_t})
+\int_X \|\dbar v_t\|^2 e^{-\phi_t}idt\wedge d\bar t.
$$
If $\F$ is harmonic, the curvature vanishes and it follows that $v_t$
is holomorphic on $X$ for any $t$ fixed. Since $u$ never vanishes we
can {\it define} a holomorphic vector field $V_t$ by
$$
-v_t=V_t\rfloor u.
$$
Almost as before we get 
$$
\dbar\dof\wedge u=\ddbar\phi_t\wedge v_t=-\ddbar\phi_t\wedge (V_t\rfloor u)
=(V_t\rfloor\ddbar\phi_t)\wedge u,
$$
which implies that
$$
V_t\rfloor i\ddbar\phi_t=i\dbar\dof.
$$
if  $\mathbf{u}$ never vanishes. This is the important point; we have been able
to trade the nonvanishing of $i\ddbar\phi_t$ for the nonvanishing of
$u$. This is where we use that the line bundle we are dealing with is
$L=-K_X$.

We also get 
the formula for  the Lie derivative of
$\ddbar\phi_t$ along $V_t$
\be
L_{V_t}\ddbar\phi_t=\partial
V_t\rfloor\ddbar\phi_t=\ddbar\dof=\frac{\partial}{\partial
  t}\ddbar\phi_t .
\ee
To be able to conclude from here we also need to prove that $V_t$
depends holomorphically on $t$. For this we will use the first term in
the curvature formula, which also has to vanish. It follows that 
$$
i\ddbar\phi \wedge \hat
u\wedge\overline{\hat u}
$$
has to vanish identically. Since this is a semidefinite form in $\hat
u$ it follows that
\be
\ddbar\phi \wedge \hat u=0.
\ee

Considering the part of this expression that contains $dt\wedge d\bar
t$ we see that
\be
\mu:=\frac{\partial^2\phi}{\partial t\partial \bar
  t}-\partial_X(\frac{\partial\phi}{\partial\bar t})(V_t)=0.
\ee

\bigskip

If $\ddbar_X\phi_t>0$, $\mu$ is easily seen to be equal to the
function $c(\phi)$ defined in the introduction, so the vanishing of
$\mu$ is then equivalent to the homogenous Monge-Amp\`ere equation.
In \cite{2Berndtsson} we showed that $\partial V_t/\partial \bar t=0$
by realizing this vector field as the complex gradient of the function
$c(\phi)$ which has to vanish if the curvature is zero. Here, where we
no longer assume strict postivity of $\phi_t$ along $X$ we have the
same problems as before to define the complex gradient. Therefore we
follow the same route as before, and start by studying  $\partial
v_t/\partial \bar t$ instead.

\bigskip

Recall that
$$
\dt v_t=\dof\wedge u +h_t
$$
where $h_t$ is holomorphic on $X$ for each $t$ fixed. As we have seen
in the beginning of this section, $v_t$ is uniquely determined, and it
is not hard to see that it depends smoothly on $t$ if $\phi$ is
smooth. Differentiating
with respect to $\bar t$ we obtain
$$
\dt\frac{\partial v_t}{\partial \bar t}=\left [\frac{\partial^2\phi}{\partial
  t\partial \bar t}-\partial_X(\frac{\partial\phi}{\partial\bar
  t})(V_t)\right ]\wedge u +\frac{\partial h_t}{\partial\bar t}.
$$
Since the left hand side is automatically orthogonal to holomorphic
forms, we get that
$$
\dt\frac{\partial v_t}{\partial \bar t}=\pi_\perp(\mu)=0,
$$
since $\mu=0$ by (3.6). Again, this means that $\partial
v_t/\partial \bar t =0$ since $\partial v_t/\partial \bar
t\wedge\omega$ is still $\dbar_X$-closed, and 
the cohomological assumption implies
that $\dt$ is injective on closed forms.  

All in all, $v_t$ is holomorphic in $t$, so $V_t$ is holomorphic on
$X\times\Omega$. We can now conclude the proof in the same way as in
\cite{2Berndtsson}. Define a holomorphic vector field $\V$ on
$X\times\Omega$ by  
$$
\V:=V_t -\frac{\partial}{\partial t}  .
$$
Let $\eta$ be the form $\ddbar_X\phi_t$ on $\X$. Then formula 2.4 says
that the Lie derivative
$$
L_\V\eta=0
$$
on $X$. It follows that $\eta$ is invariant under the flow of $\V$ so
$\ddbar\phi_t$ moves by the flow of a holomorphic family of
automorphisms of $X$.

\section{ The nonsmooth case}

In the general case we can  write our metric $\phi$ as
the uniform limit of a  
sequence of smooth  metrics, $\phi^\nu$, with $i\ddbar\phi^\nu\geq
-\epsilon_\nu\omega$, where $\epsilon_\nu$ tends to zero, see section
2.3. Note also 
that in case we assume that $-K_X>0$ we can even 
approximate with metrics of strictly positive curvature. The presence
of the negative term $-\epsilon_\nu \omega$ causes some minor
notational  problems
in the estimates below. We will therefore carry out the proof under
the assumptions that  $i\ddbar\phi^\nu\geq 0$ and leave the necessary
modifications to the reader. 

\bigskip

\noindent Let $\F_\nu$ be
defined the same way as $\F$, but using the weights $\phi^\nu$
instead. Then 
$$
i\ddbar\F_\nu
$$
goes to zero weakly on $\Omega$. We  get a sequence of $(n-1,0)$
forms $v^\nu_t$, solving
$$
\dt v^\nu_t=\pi_\perp(\dot{\phi_t^\nu} u)
$$
for $\phi=\phi_\nu$. By Remark 1, we have an $L^2$-estimate for
$v^\nu_t$ in terms of the $L^2$ norm of $\dofnu$, with the constant in
the estimate independent of $t$ and $\nu$. Since $\dot{\phi_t^\nu}$ is
uniformly bounded by section 2.2, it follows that we get a
uniform bound for the $L^2$-norms of $v_t^\nu$ over all of
$X\times\Omega$. Therefore we can select a 
subsequence of $v_t^\nu$ that converges weakly to a form $v$ in
$L^2$. Since $i\ddbar\F_\nu$ tends to zero weakly, Theorem 2.1 shows
that the $L^2$-norm of $\dbar_X v^\nu$ over $X\times K$ goes to zero
for any compact $K$ in $\Omega$, so $\dbar_X v=0$. Moreover
$$
\dt_X v= \pi_\perp (\dot{\phi}u)
$$
in the (weak ) sense that 
$$
\int_{X\times\Omega}dt\wedge d\bar t\wedge v\wedge\overline{\dbar W} e^{-\phi}=
\int_{X\times\Omega}dt\wedge d\bar t\wedge
\pi_\perp(\dot{\phi}u)\wedge\overline{ W }e^{-\phi} 
$$
for any smooth form $W$ of the appropriate degree. 

\bigskip

As before this ends the argument if there are no nontrivial holomorphic vector
fields on $X$. Then $v$ must be zero, so $\dof$ is holomorphic, hence
constant. In the general case, we finish by showing that $v_t$ is
holomorphic in $t$. The difficulty is that we don't know  any
regularity of $v_t$ 
except that it lies in $L^2$, so we need to formulate 
holomorphicity weakly. We will use two elementary lemmas that we state
without proof. The first one allows us get good convergence
properties for geodesics, when the metrics only depend on the real
part of $t$ and therfore are convex with respect to $t$.

\begin{lma}
Let $f_\nu$ be a sequence of smooth convex functions on an interval in $\R$
that converge uniformly to the convex function $f$. Let $a$ be a point
in the interval such that $f'(a)$ exists. Then $f_\nu'(a)$ converge
to $f'(a)$. Since a convex function is differentiable almost
everywhere it follows that $f'_\nu$ converges to $f'$ almost
everywhere, with dominated convergence on any compact subinterval.
\end{lma}
Another technical problem that arises is that we are dealing with
certain orthogonal projections on the manifold $X$, where the weight
depends on $t$. The next lemma gives us control of how these
projections change.
\begin{lma} Let $\alpha_t$ be forms on $X$ with coefficients depending
  on $t$ in $\Omega$. Assume that $\alpha_t$ is Lipschitz with respect to $t$ as
  a map from $\Omega$ to $L^2(X)$. Let $\pi^t$ be the orthogonal
  projection on $\dbar$-closed forms with respect to the metric $\phi_t$ and the
  fixed K\"ahler metric $\omega$. Then $\pi^t(\alpha_t)$ is also
  Lipschitz, with a Lipschitz constant depending only on that of
  $\alpha$ and the Lipschitz constant of $\phi_t$ with respect to $t$.
\end{lma}
Note that in our  case, when $\phi$ is independent of the imaginary
part of $t$, we 
have control of the Lipscitz constant with respect to $t$ of $\phi_t$ , and also
by the first 
lemma uniform control of the Lipschitz constant of $\phi^\nu_t$, since
the derivatives are increasing.

\bigskip

It follows from the curvature formula that 
$$
a_\nu:=\int_{X\times\Omega'}i\ddbar\phi^\nu\wedge\hat u\wedge\overline{\hat
  u}e^{-\phi^\nu}
$$
goes to zero if $\Omega'$ is a relatively compact subdomain of
$\Omega$. Shrinking $\Omega$ slightly we assume that this actually
holds with $\Omega'=\Omega$. By the Cauchy inequality
$$
\int_{X\times\Omega}i\ddbar\phi^\nu\wedge\hat
u\wedge\overline{W}e^{-\phi^\nu}\leq
a_\nu\int_{X\times\Omega}i\ddbar\phi^\nu\wedge W\wedge\bar W e^{-\phi^\nu}
$$
if $W$ is any $(n,0)$-form. Choose $W$ to contain no
differential $dt$, so that it is an $(n,0)$-form on $X$ with
coefficients depending on $t$. Then
$$
\int_{X\times\Omega}i\ddbar\phi^\nu\wedge W\wedge\bar W e^{-\phi^\nu}=
\int_{X\times\Omega}i\ddbar_t\phi^\nu\wedge
W\wedge\bar W e^{-\phi^\nu}
$$
We now assume that $W$ has compact support. The one variable
Hörmander inequality with respect to $t$ then shows that the last integral is
dominated by
\be
\int_{X\times\Omega}|\partial^{\phi^\nu}_t W|^2e^{-\phi^\nu}.
\ee
From now we assume that $W$ is Lipschitz with respect to $t$ as a map from
$\Omega$ into $L^2(X)$. Then (4.1) is uniformly bounded, so

$$
\int_{X\times\Omega}idt\wedge d\bar
t\wedge\mu^\nu\wedge\overline{W}e^{-\phi^\nu} 
$$
goes to zero, where $\mu^\nu$ is defined as in (3.6) with $\phi$ replaced by
$\phi^\nu$. By Lemma 4.2
$$
\int_{X\times\Omega}idt\wedge d\bar
t\wedge\mu^\nu\wedge\overline{\pi_\perp W}e^{-\phi^\nu} 
$$
also goes to zero. Therefore
$$
\int_{X\times\Omega}idt\wedge d\bar
t\wedge\pi_\perp(\mu^\nu)\wedge\overline{W}e^{-\phi^\nu}. 
$$
goes to zero. Now recall that
$\pi_\perp(\mu^\nu)=\dt (\partial v_t^\nu/\partial\bar t)$ and integrate
by parts. This gives that
$$
\int_{X\times\Omega}idt\wedge d\bar
t\wedge \frac{\partial v_t^\nu}{\partial \bar
  t}\wedge\overline{\dbar_XW}e^{-\phi^\nu}  
$$ 
also vanishes as $\nu$ tends to infinity. 

\bigskip

Next we let $\alpha$ be a form of bidegree $(n,1)$ on $X\times\Omega$
that does not contain any differential $dt$. We assume it is Lipschitz
with respect to $t$ and decompose it into one part, $\dbar_X W$,  which is
$\dbar_X$-exact and 
one which is orthogonal to $\dbar_X$-exact forms. This amounts of
course to making this orthogonal decomposition for each $t$
separately, and by Lemma 4.2 each term in the decomposition is still
Lipschitz in $t$, uniformly in $\nu$. Since $v^\nu_t\wedge\omega$ is
$\dbar_X$-closed 
by construction, this holds also for $\partial v^\nu/\partial \bar
t$. By our cohomological assumption, it is also $\dbar$-exact,
and we get that
$$
\int_{X\times\Omega}idt\wedge d\bar
t\wedge \frac{\partial v_t^\nu}{\partial \bar
  t}\wedge\overline\alpha e^{-\phi^\nu}=  \int_{X\times\Omega}idt\wedge d\bar
t\wedge \frac{\partial v_t^\nu}{\partial \bar
  t}\wedge\overline{\dbar_XW}e^{-\phi^\nu}.  
$$ 
Hence

$$
\int_{X\times\Omega}dt\wedge
v_t\wedge\overline{\partial^{\phi^\nu}_t\alpha}e^{-\phi^\nu}
$$
goes to zero. 
By Lemma 4.1 we may pass to the limit here and finally get 
that
\be
\int_{X\times\Omega} dt\wedge
v_t\wedge\overline{\partial^{\phi}_t\alpha}e^{-\phi} =0,  
\ee
under the sole assumption that $\alpha$ is of compact support, and
Lipschitz in $t$. This is almost the distributional formulation of
$\dbar_t v=0$, except that $\phi$ is not smooth. But, replacing
$\alpha$ by $e^{\phi-\psi}\alpha$, where $\psi$ is another metric on
$L$, we see that  if (4.2) holds for some $\phi$, Lipschitz in $t$, it
holds for any such metric. Therefore we can replace $\phi$ in (4.2) by
some other smooth metric. It follows that $v_t$ is holomorphic in $t$
and therefore, since we already know it is holomorphic on $X$,
holomorphic on $X\times\Omega$. This completes the proof. 
\bigskip
\section{ The Bando-Mabuchi theorem.}

For $\phi_0$ and $\phi_1$, two metrics on a line bundle $L$ over $X$,
we consider their relative energy
$$
\E(\phi_0,\phi_1).
$$
This is well defined if $\phi_j$ are bounded with
$i\ddbar\phi_j\geq 0$. It has the fundamental properties that if
$\phi_t$ is smooth in $t$ for $t$ in $\Omega$, then
$$
\frac{\partial}{\partial t}\E(\phi_t,\phi_1)=\int_X \dof
(i\ddbar\phi_t)^n/\text{Vol}(L)
$$
and
$$
i\ddbar_t\E(\phi_t,\phi_1)=p_*((i\ddbar_{X,t}\phi)^{n+1})/\text{Vol}(L)=idt\wedge
d\bar t\int_X
c(\phi_t)(i\ddbar_X \phi_t)^n/\text{Vol}(L),
$$
where $p$ is the projection map from $X\times\Omega$ to $\Omega$.
Here Vol($L$) is the normalizing factor
$$
\text{Vol}(L)=\int_X(i\ddbar_X \phi)^n,
$$
chosen so that the derivative of $\E$ becomes 1 if $\phi_t=\phi+t$. 
If the family is only bounded, these formulas hold in the sense of
distributions. In particular, if $\phi$ solves the homogenuous
Monge-Amp\`ere equation, so that $(i\ddbar_{X,t}\phi)^{n+1}=0$ or
equivalently $c(\phi)=0$, then $\E(\phi_t,\phi_1)$ is harmonic in
$t$. Hence this function is linear along geodesics.

\bigskip

Let now 
$$
\G(t)=\F(t)-\E(\phi_t,\psi)
$$
where $\psi$ is arbitrary. Then $\phi_0$ solves the K\"ahler-Einstein
equation if and only if $\G'(0)=0$  for any smooth curve $\phi_t$. If
$\phi_0$ and $\phi_1$ are two K\"ahler-Einstein metrics we connect them
by a geodesic $\phi_t$ (a continuous geodesic will be enough). Now
$\phi_t$ depends only on the real part of $t$ so $\G$ is convex. We
claim that since
both end points are K\"ahler-Einstein metrics, 0 and 1 are stationary
points for $\G$. This would be immediate if the geodesic were smooth,
but it is not hard to see that it also holds if the geodesic is only
bounded, with boundary behaviour as described in section 2.2.  The
function $\F$ is convex, hence has onesided derivatives at the
endpoints, and using the convexity of $\phi$ with respect to $t$ one
sees that they equal
$$
\int \dof e^{-\phi}/\int e^{-\phi}
$$
(where $\dof$ now stands for the onesided derivatives). The function
$\E(\phi_t,\psi)$ is linear so its distributional derivative
$$
\int_X \dof
(i\ddbar\phi_t)^n/\text{Vol}(L)
$$
is constant and simple convergence theorems for the Monge-Amp\`ere
operator show that it is equal to its values at the endpoints. Hence
both end points are critical points for $\G$ and the convexity implies
that $\G$ is constant so $\F$ is  linear.

By Theorem
1.2 $\ddbar\phi_t$ are related via a holomorphic family of
automorphisms. In particular $\ddbar\phi_0$ and $\ddbar\phi_1$ are
related via an automorphism which is homotopic to the identity, which
is the content of the Bando-Mabuchi theorem.

\section{A generalized Bando-Mabuchi theorem}

\subsection{A variant of Theorem 1.2 for unbounded metrics}
One might ask if Theorem 1.2 is valid under even more general
assumptions. A minimal requirement is of course that $\F$ be finite,
or in other words that $e^{-\phi_t}$ be integrable. For all we know
Theorem 1.2 might be true in this generality, but here we
will limit ourselves to  the following situation:

\bigskip
\noindent
 Let $t\rightarrow \tau_t$ be a curve of singular metrics
  on $L=-K_X$ that can be written
$$
\tau_t=\phi_t +\psi
$$
where $\psi$ is a metric on an $\R$-line bundle $S$ and $\phi_t$ is a
curve of metrics on $-(K_X+S)$ such that:

\medskip

(i) $ \phi_t$ is bounded  and
only depends on $\Re t$.

\medskip

(ii) $e^{-\psi}$ is integrable and  $\psi$ does not depend on $t$

and

\medskip

(iii) 
$i\ddbar_{t,X}(\tau_t)\geq 0$. 

\medskip

\begin{thm} Assume that $-K_X\geq 0$ and that $H^{0,1}(X)=0$. Let
  $\tau_t=\phi_t+\psi$ be a 
  curve of metrics on $-K_X$ 
  satisfying  (i)-(iii). Assume that
$$
\F(t)=-\log\int_X e^{-\tau_t}
$$
is affine. Then there is a holomorphic vector field $V$ on $X$ with
flow $F_t$ such that
$$
F_t^*(\ddbar\tau_t)=\ddbar\tau_0.
$$
\end{thm}

The proof of this theorem is almost the same as the proof of Theorem
1.2. The main thing to be checked is that for $\tau=\tau^\nu$ a
sequence of smooth metrics decreasing to $\tau$ we can still solve the
equations 
$$
\partial^{\tau_t} v_t=\pi_\perp (\dot{\tau_t} u)
$$
with an $L^2$ -estimate independent of $t$ and $\nu$. 
\begin{lma}
Let $L$ be a holomorphic line bundle over $X$ with a metric $\xi$
satisfying $i\ddbar\xi\geq 0$. Let $\xi_0$ be a smooth metric on $L$
with $\xi\leq \xi_0$, and assume
$$
I:=\int_X e^{\xi_0-\xi}<\infty.
$$
Then there is a constant $A$, only depending on $I$ and $\xi_0$ (
not on $\xi$!) such that if $f$ is a $\dbar$-exact $L$ valued
$(n,1)$-form with 
$$
\int |f|^2 e^{-\xi}\leq 1
$$
there is a solution $u$ to $\dbar u=f$ with 
$$
\int_X |u|^2 e^{-\xi}\leq A.
$$
(The integrals are understood to be taken with respect to some
arbitrary smooth volume form.)
\end{lma}  
\begin{proof}
The assumptions imply that
$$
\int |f|^2 e^{-\xi_0}\leq 1.
$$
Since $\dbar$ has closed range for $L^2$-norms defined by smooth
metrics, we can solve $\dbar u=f$ with
$$
\int |u|^2 e^{-\xi_0}\leq C
$$
for some constant depending only on $X$ and $\xi_0$. Choose a
collection of coordinate balls $B_j$ such that $B_j/2$ cover $X$. In
each $B_j$ solve $\dbar u_j=f$ with
$$
\int_{B_j} |u_j|^2 e^{-\xi}\leq C_1\int_{B_j} |f|^2 e^{-\xi}\leq C_1,
$$
$C_1$ only depending on the size of the balls. Then $h_j:=u-u_j$ is
holomorphic on $B_j$ and 
$$
\int_{B_j}|h_j|^2 e^{-\xi_0}\leq C_2,
$$
so
$$
\sup_{B_j/2}|h_j|^2 e^{-\xi_0}\leq C_3.
$$
Hence
$$
\int_{B_j/2}|h_j|^2 e^{-\xi}\leq C_3 I
$$
and therefore
$$
\int_{B_j/2}|u|^2 e^{-\xi}\leq C_4 I.
$$
Summing up we get the lemma.
\end{proof}

By the discussion in section 2, the assumption that $-K_X\geq 0$
implies that we can write $\tau_t$ as a limit of a decreasing sequence
of smooth metrics $\tau_t^\nu$ with
$$
i\ddbar\tau_t^\nu\geq -\epsilon_\nu \omega
$$
where $\epsilon_\nu$ tends to zero. 
Applying the lemma  to $\xi=\tau^\nu_t$ and $\xi_0$ some arbitrary smooth
metric we see that we have uniform estimates for solutions of the
$\dbar$-equation, independent of $\nu$ and $t$. By remark 2, section
3, the same holds for the adjoint operator, which means that we can
construct $(n-1,0)$-forms $v_t^\nu$ just as in section 3, and the
proof of Theorem 6.1 then continues as in section 3. 

\bigskip

\subsection{Yet another version}
We also briefly describe yet another situation where the same
conclusion as in Theorem 6.1 can be drawn even though we do not assume
that $-K_X\geq 0$. The assumptions are very particular, and it is not
at all clear that they are optimal, but they are chosen to fit with
the properties of desingularisations of certain singular varieties. We
then assume instead that $-K_X$ can be 
decomposed
$$
-K_X = -(K_X +S) +S
$$
where $S$ is the $\R$-line bundle corresponding to a klt -divisor
$\Delta\geq 0$ and
we assume $-(K_X+S)\geq 0$. We moreover assume that the underlying
variety of $\Delta$ is a union of
smooth hypersurfaces with simple normal crossings. We then look at curves 
$$
\tau_t=\phi_t +\psi
$$
where $i\ddbar_{t, X}\phi_t\geq 0$ and $\psi$ is a fixed  metric on $S$
satisfying $i\ddbar\psi=[\Delta]$. We claim that the conclusion of
Theorem 6.1 holds in 
this situation as well. The difference as compared to our previous
case is that we do not assume that $\tau_t$ can be approximated by a
decreasing sequence of metrics with almost positive curvature.  For
the proof we approximate $\phi_t$ by a 
decreasing sequence of smooth metrics $\phi^\nu$ satisfying
$$
i\ddbar\phi_t^\nu\geq -\epsilon_\nu \omega.
$$
As for $\psi$ we approximate it following the scheme at the end of
section 2 by a sequence satisfying
$$
i\ddbar\psi^\nu\geq -C\omega
$$
and 
$$
i\ddbar\psi^\nu\geq -\epsilon_\nu \omega
$$
outside of any neighbourhood of $\Delta$. Then let
$\tau^\nu_t=\phi^\nu_t+\psi^\nu$. Now consider the curvature
formula (3.2)
\be
\langle\Theta^\nu u_t,u_t\rangle_t= p_*(c_n i\ddbar\tau^\nu_t \wedge \hat
u\wedge\overline{\hat u} e^{-\tau^\nu_t})  
+\int_X \|\dbar v^\nu_t\|^2 e^{-\tau^\nu_t}idt\wedge d\bar t
\ee
We want to see that the second term in the right hand side tends to
zero given that the 
curvature $\Theta^\nu$ tends to zero, and the problem is that the
first term on the right hand side has a negative part. However,
$$
 p_*(c_n i\ddbar\tau^\nu_t \wedge \hat
u\wedge\overline{\hat u} e^{-\tau^\nu_t}) 
$$
can for any $t$ be estimated from below by 
\be
-\epsilon_\nu \|\hat u\|^2 -C\int_U |v^\nu_t|^2 e^{-\tau^\nu}
\ee
where $U$ is any small neighbourhood of $\Delta$ if we choose $\nu$
large. This means, first, that we still have at least a uniform upper
estimate on $\dbar v^\nu_t$. This, in turn gives by the technical lemma below
that the $L^2$-norm of $v^\nu_t$ over a small neighbourhood of
$\Delta$ must be small if the neighbourhood is small. Shrinking the
neighbourhood as $\nu$ grows we can then arrange things so that the
negative part in the right hand side goes to zero. Therefore the
$L^2$-norm of $\dbar v^\nu_t$ goes to zero after all, and after that
the proof proceeds as before. We collect this in the next theorem.
\begin{thm} Assume that $-(K_X+S)\geq 0$ and that $H^{0,1}(X)=0$. Let
  $\tau_t=\phi_t+\psi$ be a 
  curve of metrics on $-K_X$ 
where 

(i) $\phi_t$ are metrics on $-(K_X+S)$ with $i\ddbar\phi_t\geq 0$,

and 

(ii) $\psi$ is a metric on $S$ with $i\ddbar\psi=[\Delta]$, where
$\Delta$ is a klt divisor with simple normal crossings. 

Assume that
$$
\F(t)=-\log\int_X e^{-\tau_t}
$$
is affine. Then there is a holomorphic vector field $V$ on $X$ with
flow $F_t$ such that
$$
F_t^*(\ddbar\tau_t)=\ddbar\tau_0.
$$
\end{thm}
We end this section with the technical lemma used above.
\begin{lma} The term
$$
\int_U |v^\nu_t|^2 e^{-\tau^\nu}
$$
in (6.2) can be made arbitrarily small if $U$ is a sufficiently small
neighbourhood of $\Delta$
\end{lma}
\begin{proof}
Covering $\Delta$ with a finite number of polydisks, in which the
divisor is a union of coordinate hyperplanes,  it is enough to prove
the following statement:

Let $P$ be the unit  polydisk in $\C^n$ and let $v$ be a compactly
supported function in $P$. Let 
$$
\psi_\epsilon=\sum \alpha_j\log( |z_j|^2 +\epsilon)
$$
where $0\leq \alpha_j<1$. Assume
$$
\int_P (|v|^2 +|\dbar v|^2)e^{-\psi}\leq 1.
$$
Then for $\delta>>\epsilon$
$$
\int_{\cup\{|z_j|\leq \delta\}} |v|^2 e^{-\psi_\epsilon} \leq c_\delta
$$
where $c_\delta$ tends to zero with $\delta$.

To prove this we first  estimate the integral over $ |z_1|\leq \delta$
 using the one 
variable Cauchy formula in the first variable
$$
v(z_1, z')= \pi^{-1}\int v_{\bar\zeta_1}(\zeta_1,z')/(\zeta_1 -z_1)
$$
which gives
$$
|v(z_1, z')|^2 \leq C \int |v_{\bar\zeta_1}(\zeta_1,z')|^2/|\zeta_1
-z_1|.
$$
Then multiply by $(|z_1|^2 +\epsilon)^{-\alpha_1}$ and integrate with
respect to $z_1$ over $|z_1|\leq \delta$. Use the estimate
$$
\int_{|z_1|\leq \delta} \frac{1}{(|z_1|^2
  +\epsilon)^{\alpha_1}|z_1-\zeta_1|}\leq  c_\delta
(|\zeta_1|^2+\epsilon)^{-\alpha_1},
$$
multiply by $\sum_2^n \alpha_j\log( |z_j|^2 +\epsilon)$ and integrate
with respect to $z'$. Repeating the same argument for $z_2, ..z_n$ and
summing up we get the required estimate.

\end{proof}
\subsection{A generalized Bando-Mabuchi theorem}
As pointed out to me by Robert Berman, Theorems 6.1 and 6.3  lead to  versions
of the Bando-Mabuchi theorem for 'twisted K\"ahler-Einstein
equations', \cite{Szekelyhidi}, \cite{2Berman}, and
\cite{2Donaldson}. Let $\theta$ be 
a positive $(1,1)$-current that can be written
$$
\theta=i\ddbar\psi 
$$
with $\psi$ a metric on a $\R$-line bundle $S$. The twisted
K\"ahler-Einstein equation is 
\be
\text{Ric}(\omega) =\omega +\theta,
\ee
for a K\"ahler metric $\omega$ in the class $c[-(K_X+ S)]$. Writing
$\omega=i\ddbar\phi$, where $\phi$ is a metric on the $\R$-line bundle 
 $F:=-(K_X+S)$, this is equivalent to
\be
(i\ddbar\phi)^n= e^{-(\phi+\psi)},
\ee
after adjusting constants. 

To be able to apply Theorems 6.1 and 6.2 we need to assume that $e^{-\psi}$ is
integrable. By this we mean that representatives with respect to a
local frame are integrable. When $\theta=[\Delta]$ is the current defined by a
divisor, it means that the divisor is klt.

Solutions $\phi$ of (6.2) are now critical points of the function
$$
\G_\psi(\phi):=-\log\int e^{-(\phi+\psi)} -\E(\phi, \chi)
$$
where $\chi$ is an arbitary metric on $F$. Here  $\psi$ is fixed and we
let the variable $\phi$ range over bounded metrics with
$i\ddbar\phi\geq 0$. If $\phi_0$ and $\phi_1$ are two critical points,
it follows from the discussion in section 2  that
we can connect them with a bounded  geodesic $\phi_t$. Since $\E$ is
affine along 
the geodesic it follows that 
$$
t\rightarrow -\log\int e^{-(\phi_t+\psi)} 
$$
is affine along the geodesic and we can apply Theorem 6.1. 
\begin{thm} Assume that $-K_X$ is semipositive and that
  $H^{0,1}(X)=0$. 
Assume that  $i\ddbar\psi=\theta$, where
$e^{-\psi}$ is integrable. Let
$\phi_0$ and $\phi_1$ be two bounded solutions of equation (6.3)
with $i\ddbar\phi_j\geq 0$. Then there is a holomorphic automorphism, $F$,
of $X$, homotopic to the identity, such that 
$$
F^*(\ddbar\phi_1)=\ddbar\phi_0
$$
and
$$
F^*(\theta)=\theta.
$$
\end{thm}

\begin{proof} By Theorem 6.1 there is an $F$ such that
$$
F^*(\ddbar\phi_1+\ddbar\psi)=\ddbar\phi_0+\ddbar\psi
$$
so we just need to see that $F$ preserves $\theta=i\ddbar\psi$. But this
follows since $\omega^j:=i\ddbar\phi_j$ solves (6.1) and
$F^*(\text{Ric}(\omega^1))= \text{Ric}(F^*(\omega^1))$. Thus
$$
\omega^1 +\theta=\text{Ric}(\omega^1)
$$
implies
$$
\omega^0+F^*(\theta)=\text{Ric}(\omega^0)=\omega^0 +\theta,
$$
and we are done.

\end{proof}
\begin{remark}
 Note that in case $\theta$ is  strictly positive we even
 get absolute uniqueness. This follows from the proof of Theorem 6.1
 since both $\phi_t$ and $\phi_t+\psi$ must be geodesics, which forces
 $\phi_t$ to be linear in $t$ if $i\ddbar\psi>0$. Certainly the
 assumption on strict positivity can be considerably relaxed here, see
 the end of the next section for a comment on this.
\end{remark}
In the same way we get from Theorem 6.3 
\begin{thm} Assume that $-K_X=-(K_X+S) +S $ where $-(K_X+S)$ is
  semipositive and $S$ is the $\R$-line bundle corresponding to a klt
  divisor $\Delta\geq 0$ with simple normal crossings. Assume also that 
  $H^{0,1}(X)=0$. 
  Let
$\phi_0$ and $\phi_1$ be two bounded solutions of equation (6.2) with
$\theta=[\Delta]$ and 
with $i\ddbar\phi_j\geq 0$. Then there is a holomorphic automorphism, $F$,
of $X$, homotopic to the identity, such that 
$$
F^*(\ddbar\phi_1)=\ddbar\phi_0
$$
and
$$
F^*([\Delta])=[\Delta].
$$
\end{thm}

 \section{ A concluding (wonkish) remark on complex
  gradients} 

The curvature formula in Theorem 3.1 is based on a particular choice
of the auxiliary $(n-1,0)$ form $v_t$ as the solution of an equation
$$
\dt v_t=\pi_\perp(\dof u_t).
$$

In the case when $\phi_t$ is smooth and $i\ddbar_X\phi_t>0$ one could
alternatively choose $\tilde v_t$ as 
$$
\tilde v_t=V_t\rfloor u,
$$
where $V_t$ is the complex gradient of $\dof$ defined by
$$
V_t\rfloor \ddbar_X\phi_t=\dbar\dof.
$$
This leads to a different formula for the curvature which is the one
used in \cite{2Berndtsson}:
\be
\langle\Theta^E u,u\rangle=\int_{X_t} c(\phi) |u|^2 e^{-\phi} +\langle
  (\Box +1)^{-1}\dbar \tilde v_t,\tilde v_t\rangle,
\ee
where $\Box$ is the $\dbar$-Laplacian for the metric
$i\ddbar_X\phi_t$. The relation between the two formulas is discussed
in \cite{3Berndtsson} in the more general setting of a nontrivial
fibration. At any
rate, the two choices $v_t$ and $\tilde v_t$ coincide in case the
curvature vanishes, as we have seen in section 3. 

Of course the definition of $\tilde v_t$ makes no sense in our more
general setting since we have no metric on $X$ to help us define a
complex gradient. Nevertheless, the methods of section 3 can perhaps
be seen as giving a way to define a 'complex gradient' in a nonregular
situation. We formulate the basic principle in the next proposition.
\begin{prop} Let $L$ be a holomorphic line bundle over the compact
  K\"ahler manifold $X$, and let $\phi$ be a smooth metric on $L$, not
  necessarily with positive 
  curvature. Assume $V$ is a holomorphic vector field on $X$ such that
$$
V\rfloor \ddbar\phi=0.
$$
Then $V=0$ provided that 
$$
H^{(0,1)}(X, K_X+L)=0
$$
and
$$
H^{0}(X, K_X+L)\neq 0.
$$
\end{prop}
\begin{proof} We follow the arguments in section 3. Let $u$ be a
  global holomorphic section of $K_X+L$, and put
$$
v:=V\rfloor u.
$$
Then $v$ is a holomorphic $(n-1,0)$-form and
$$
\ddbar\phi\wedge v=-(V\rfloor\ddbar\phi)\wedge u=0.
$$
Hence
$$
\dbar\partial^\phi v=-\partial^\phi\dbar v=0.
$$
Put $\alpha=v\wedge \omega$ where $\omega$ is the K\"ahler form. Then
$\alpha $ is a smooth, $\dbar$-closed $(n,1)$-form solving
$$
\dbar\dbar^*_\phi\alpha=0.
$$
This means that $\dbar^*_\phi\alpha$ is a holomorphic, hence smooth
$(n,0)$-form. Integrating by parts we get
$$
|\dbar^*_\phi\alpha|^2=0
$$
Since we have assumed $H^{(n,1)}=0$, $\alpha=\dbar g$ for some
$g$. Then
$$
\|\alpha\|^2=\langle\alpha,\dbar g\rangle=0
$$
so $v$ and hence $V$ are 0.
\end{proof}

\bigskip

\noindent This means that holomorphic solutions of 
$$
V\rfloor \ddbar\phi=\dbar\chi
$$
are unique, if they exist. 

\bigskip

Let us finally compare this to our first uniqueness result for twisted
K\"ahler-Einstein equations, Theorem 6.5,  and the remark immediately after it
(Remark 4). There we noted that in case the twisting term $\theta$ is
strictly positive, the automorphism $F$ must be the identity, so that
we even get absolute uniqueness, and not just uniqueness up to a
holomorphic automorphism. A considerably more general statement
follows from Proposition 7.1: For absolute uniqueness it suffices to
assume that some multiple of the $\R$ bundle $S$ satisfies the
cohomological assumptions in Proposition 7.1, $H^0(X, K_X+mS)\neq 0$
and $H^1(X, K_X+mS)=0$. This is certainly the case (by Kodaira
vanishing) if $S>0$, even if $\theta$ itself is not assumed
positive. Of course it also holds in many other cases that are not
covered by Kodaira's theorem. 

In this connection, notice also that some kind of regularity of 
$\phi$ in Proposition 7.1 is necessary, since for completely general
metrics the meaning 
of the operators $\partial^\phi$ and $\dbar^*_\phi$ becomes
unclear. This is not just a technical problem. The vector field
$z\partial/\partial z$ vanishes at $z=0$ and $z=\infty$ on the Riemann
sphere. But, the divisor $\{0\}\cup\{\infty\}$ is certainly ample.

\def\listing#1#2#3{{\sc #1}:\ {\it #2}, \ #3.}

\end{document}